\documentclass[12pt,notitlepage, reqno]{amsart}
\usepackage{graphicx}
\usepackage{amsmath, amsfonts, amssymb, amsthm,dsfont, mathtools}
\usepackage[myheadings]{fullpage}
\usepackage{enumitem}
\usepackage{xcolor}
\usepackage[symbol]{footmisc}
\usepackage{subcaption}
\usepackage[hyperfootnotes=false]{hyperref}

\newcommand{\R}{\mathbb R}

\newcommand{\N}{\mathbb N}

\newcommand{\D}{\mathcal D}
\newcommand{\s}{\mathcal S}
\newcommand{\1}{\mathbf 1}

\newcommand{\Exp}{\operatorname{Exp}}

\newcommand{\Gam}{\operatorname{Gamma}}
\newcommand{\Binom}{\operatorname{Binomial}}
\newcommand{\Bernoulli}{\operatorname{Bernoulli}}
\newcommand{\Pareto}{\operatorname{Pareto}}
\newcommand{\Levy}{\operatorname{L\acute{e}vy}}
\newcommand{\Weibull}{\operatorname{Weibull}}
\newcommand{\SkewNormal}{\operatorname{SkewNormal}}
\newcommand{\ChiSquared}{\chi^2}
\newcommand{\Skew}{\operatorname{Skew}}

\DeclareMathOperator*{\argmin}{arg\,min}

\newtheorem{theorem}{Theorem}[section]

\newtheorem{corollary}[theorem]{Corollary}
\newtheorem{lemma}[theorem]{Lemma}

\newtheorem{proposition}[theorem]{Proposition}

\theoremstyle{definition}
\newtheorem{definition}[theorem]{Definition}

\theoremstyle{remark}
\newtheorem{example}[theorem]{Example}
\newtheorem{remark}[theorem]{Remark}

\numberwithin{theorem}{section}
\numberwithin{equation}{section}

\numberwithin{theorem}{section}

\begin{document}


\title{Extensions of True Skewness for \\ Unimodal Distributions}
\author{Yevgeniy Kovchegov}
\address{Oregon State University}
\email{kovchegy@math.oregonstate.edu}
\author{Alex Negr\'on}
\address{Illinois Institute of Technology}
\email{alexnegron18@gmail.com}
\author{Clarice Pertel}
\address{Cornell University}
\email{cep87@cornell.edu}
\author{Christopher Wang}
\address{Columbia University}
\email{cyw2124@columbia.edu}

\date{\today}

\begin{abstract}
	A 2022 paper \cite{Kovchegov21} introduced the notion of true positive and negative skewness for continuous random variables via Fr\'echet $p$-means. In this work, we find novel criteria for true skewness, establish true skewness for the Weibull, L\'evy, skew-normal, and chi-squared distributions, and discuss the extension of true skewness to discrete and multivariate settings. Furthermore, some relevant properties of the $p$-means of random variables are established.
\end{abstract}


\maketitle
\markboth{Kovchegov, Negr\'on, Pertel, Wang}{Extensions of True Skewness}


\tableofcontents


\section{Introduction}\label{S:Introduction}

	\noindent A commonly accepted measure for the skewness of a probability distribution is given by 
	\emph{Pearson's moment coefficient of skewness}, or the standarized third central moment 
	$$\gamma := E\left[\left(\frac{X-\mu}{\sigma}\right)^3\right],$$
	where $X$ has mean $\mu$ and variance $\sigma^2$. Usually, we say that a distribution is positively skewed if $\Skew[X]>0$ and negatively skewed if $\Skew[X]<0$. It is also expected that positively skewed distributions satisfy the \emph{mean-median-mode inequalities}
	$$\text{mode} < \text{median} < \text{mean}$$
	while negatively skewed distributions satisfy the reverse inequalities
	$$\text{mean} < \text{median} < \text{mode}.$$
	However, a distribution with positive moment coefficient of skewness does not always satisfy the mean-median-mode inequalities. For instance, the Weibull distribution with shape parameter $3.44<\beta<3.60$ has positive moment coefficient of skewness but satisfies the reversed mean-median-mode inequalities, corresponding to negative skew \cite{Groeneveld86}. (For other counterexamples, see \cite{Abadir05}.) This discrepancy is important when comparing Pearson's moment coefficient with other measures of skewness, such as \emph{Pearson's first skewness coefficient}
	$$\frac{\text{mean}-\text{mode}}{\text{standard deviation}}$$\\
	and \emph{Pearson's second skewness coefficient}\\
	$$3\times\frac{\text{mean}-\text{median}}{\text{standard deviation}}.$$\\
	The direction of skewness can therefore be inconsistent between different skewness measures.
	In 2022, Kovchegov \cite{Kovchegov21} introduced the notion of \emph{true positive and negative skewness} to unify Pearson's coefficients in determining the sign of skewness. It relies on the idea that for a positively skewed distributions, the left part of the distribution should stochastically dominate the right part, resulting in a left tail that ``spreads shorter'' and a right tail that ``spreads longer.'' We use a class of centroids known as $p$-means. See \cite{Kovchegov21}.
	
	\begin{definition}[Kovchegov \cite{Kovchegov21}, 2022]
		For $p\in[1,\infty)$ and random variable $X$ with finite $(p-1)$-st moment, the \textbf{$p$-mean}, of $X$ is the unique solution $\nu_p$ of the equation
		\begin{equation}\label{nup def 2}
			E[(X-\nu_p)_+^{p-1}] = E[(\nu_p-X)_+^{p-1}].
		\end{equation}
		If moreover $X$ has finite $p$-th moment, the notion of \textbf{$p$-mean} in \eqref{nup def 2} is equivalent to the \textbf{Fr\'echet 
		$p$-mean} defined as
		\begin{equation}\label{nup def 1}
			\nu_p = \argmin_{a\in\R} E|X-a|^p.
		\end{equation}
		For $p>1$, the uniqueness of the \textbf{$p$-mean} $\nu_p$ follows from the fact that $$E[(a-X)_+^{p-1}]-E[(X-a)_+^{p-1}]$$ is a strictly increasing 
		continuous function of $a$. Occasionally, we write $\nu(p)$ to emphasize that $\nu_p$ is a function
		of $p$.
	\end{definition}
	
	\noindent Notice that identically distributed random variables have the same $p$-means and that $\nu_1$ and $\nu_2$ correspond to the median and mean of the distribution, respectively. For a random variable $X$ let
	$$\D_X:=\big\{p\ge1:E\big[|X|^{p-1}\big]<\infty\big\}$$
	be the domain of $\nu_p$. Notice that $p$ is a real number; it does not have to be an integer. If $X$ has a unique mode, then we denote it by $\nu_0$. In this case, we include $0$ in the set $\D_X$. We omit the subscript $X$ when the random variable is unambiguous.
	
	\begin{remark}
		For distributions which are ``nice enough,'' the domain $\D_X$ can be extended in a well-defined way
		to include values of $p$ in the interval $(0,1)$. The conditions for which this is permissible is 
		discussed in \cite{Kovchegov21}. In this work, we limit our consideration to $p\ge1$.
	\end{remark}

	\begin{definition}[Kovchegov \cite{Kovchegov21}, 2022]
		We say a random variable $X$ (resp. its distribution and density) witha uniquely defined median $\nu_1$ is \textbf{truly positively skewed} if
		$\nu_p$ is a strictly increasing function of $p$ in $\D$, provided $\D$ has non-empty interior.
		Analogously, $X$ is \textbf{truly negatively skewed} if $\nu_p$ is a strictly decreasing function of $p$
		in $\D$.
	\end{definition}

	\begin{remark}\label{TMPS}
		It is possible that, for a unimodal distribution, $\nu_p$ is strictly increasing only on
		$\D\setminus\{0\}$ and that there exists $p\in\D\setminus\{0\}$ such that $\nu_p<\nu_0$. 
		Kovchegov \cite{Kovchegov21} differentiates between this case, which is referred to in 
		that work as true positive skewness, and the case where $\nu_p>\nu_0$ holds for all $p \in \D$, which is 
		referred to as the stronger \emph{true mode positive skewness}. Because we consider only 
		unimodal distributions, for simplicity we take ``true positive skewness'' to mean true mode 
		positive skewness in the sense of Kovchegov \cite{Kovchegov21}, unless explicitly mentioned otherwise.
	\end{remark}

	\noindent 
	Consider $X$ that has unique mode and median.
	Naturally, Pearson’s first skewness coefficient is positive if and only if $\nu_2>\nu_0$ and Pearson’s second skewness coefficient is positive if and only if $\nu_2>\nu_1$. 
	Additionally, it was noticed in \cite{Kovchegov21} that Pearson's moment coefficient of skewness $\gamma$ is positive if and only if $\nu_4>\nu_2$. 
	Now, true positive skewness guarantees $\nu_0<\nu_1<\nu_2<\nu_4$ provided finiteness of the corresponding moments. Thus, Pearson's first and second skewness coefficients and Pearson's moment coefficient of skewness coincide in sign. Therefore, the direction of skewness is unified across several different criteria under true positive skewness.

	\begin{remark}\label{rem:infmoment}
	An additional advantage of the notion of true positive skewness is that it allows us to characterize the skewness of distributions that have infinite integer moments. Indeed, each of Pearson's skewness coefficients requires at least a finite mean, which excludes a large class of heavy-tailed distributions from the classical study of skewness.
	\end{remark}

\subsection{Existing criteria for true skewness}\label{Ss:Existing Criteria}

	In general, demonstrating the true skewness of an arbitrary distribution requires a detailed analysis of the solutions of a class of integral equations \eqref{nup def 2}. Such analysis is simplified by introducing arguments from stochastic ordering.
	
	\begin{definition}\label{stochastic dominance}
		For random variables $X$ and $Y$, we say that $X$ \textbf{stochastically dominates} $Y$ 
		(resp. the distribution of $X$ stochastically dominates the distribution of $Y$) if the 
		cumulative distribution function $F_X$ of $X$ is majorized by the 
		cumulative distribution function $F_Y$ of $Y$, i.e. if $F_X(x)\le F_Y(x)$ holds for every $x\in\R$. We say 
		that the stochastic dominance is \emph{strict} if there exists a point $x_0$ for which 
		$F_X(x_0)<F_Y(x_0)$.
	\end{definition}

	If $X$ is a continuous random variable with a density function $f$, then we may rewrite \eqref{nup def 2} as an equality of integrals, the quantities of which we define as a normalizing term $H_p$, i.e.,
	\begin{equation}\label{nup def 2 integral}
		H_p := \int_0^{\nu_p-L} y^{p-1}f(\nu_p-y)\,dy = \int_0^{R-\nu_p} y^{p-1}f(\nu_p+y)\,dy,
			\qquad p\in\D\setminus\{0\},
	\end{equation}
	where $X$ has support on the possibly infinite interval $(L,R)$. The following result establishes the relationship between true positive skewness and stochastic dominance of the left and right parts of a distribution. 
	
	\begin{theorem}[Kovchegov \cite{Kovchegov21}, 2022]\label{Kovchegov T1}
		Let $X$ be a continuous random variable supported on $(L,R)$ with density function $f$. 
		For fixed $p\in\D$, if the distribution with density function $y\mapsto\frac1{H_p}y^{p-1}
		f(\nu_p+y)\1_{(0,R-\nu_p)}(y)$ exhibits strict stochastic dominance over the distribution
		with density function $y\mapsto\frac1{H_p}y^{p-1}f(\nu_p-y)\1_{(0,\nu_p-L)}(y)$, then 
		$\nu(\cdot)$ is increasing at $p$.
	\end{theorem}
	 
	\begin{remark}
		In particular, Theorem \ref{Kovchegov T1} explains why stochastic dominance of the density function $\frac1{H_p}y^{p-1}f(\nu_p+y)\1_{(0,R-\nu_p)}(y)$
		representing the right part of $f$
		over the density function $\frac1{H_p}y^{p-1}f(\nu_p-y)\1_{(0,\nu_p-L)}(y)$ representing the left part of $f$ for all $p\in\D$ implies 
		positive Pearson's moment coefficient of skewness ($\nu_4>\nu_2$) and positive Pearson’s second skewness coefficient ($\nu_2>\nu_1$).
		Importantly, Theorem \ref{Kovchegov T1} provides a mathematical justification why skewness to the right, as expressed through stochastic domination,
		implies true positive skewness, yielding positivity of the above listed Pearson's coefficients of skewness.		
 	\end{remark}

	\begin{remark}
		The crux of the proof of Theorem \ref{Kovchegov T1} is that $\nu_p$ is increasing if and only if
		\begin{equation}\label{dnup numerator}
			\int_0^{R-\nu_p}y^{p-1}\log y\ f(\nu_p+y)\,dy 
				- \int_0^{\nu_p-L}y^{p-1}\log y\ f(\nu_p-y)\,dy > 0.
		\end{equation}
		Therefore, weaker versions of stochastic ordering between the left and right parts 
		actually suffice for Theorem \ref{Kovchegov T1}, in particular the concave ordering (see,
		e.g., \cite{Muller02, Rojo13}).
	\end{remark}
	
	\noindent We will frequently use the following criterion for true positive skewness.
	
	\begin{lemma}[Kovchegov \cite{Kovchegov21}, 2022]\label{Kovchegov L2}
		Let $X$ be a continuous random variable supported on the possibly infinite interval $(L,R)$ 
		with density function $f$. For fixed $p\in\D$, suppose there exists $c>0$ such that 
		\begin{enumerate}[label=(\alph*)]
			\item $f(\nu_p-c)=f(\nu_p+c)$;
			\item $f(\nu_p-x)>f(\nu_p+x)$ for $x\in(0,c)$; and
			\item $f(\nu_p-x)<f(\nu_p+x)$ for $x>c$.
		\end{enumerate}
		If $\nu_p-L\le R-\nu_p$, or if $L=-\infty$ and $R=\infty$, then $\nu(\cdot)$ is increasing at $p$.
	\end{lemma}
	
	\noindent The following is a special case.
	
	\begin{proposition}[Kovchegov \cite{Kovchegov21}, 2022]\label{Kovchegov P2}
		If $f$ is strictly decreasing on its support, then $X$ is truly positively skewed.
	\end{proposition}
	
	The requirement of strict monotonicity in the proposition can be relaxed, which will be necessary when 
	we consider uniform mixtures in a later section.
	
	\begin{proposition}\label{relaxed Kovchegov P2}
		 If $f$ is non-increasing on its support $(L,R)$, and there exist any two points $y_1,y_2\in(L,R)$, 
		 $y_1<y_2$, such that $f(y_1)>f(y_2)$, then $X$ is truly positively skewed.
	\end{proposition}
	
	The strict inequalities in Lemma \ref{Kovchegov L2} can be relaxed in a similar manner.


\section{Results}

\subsection{Properties of $p$-means}\label{Ss:Properties}

In this section, we establish several simple but fundamental properties of $p$-means and their behavior.
Here and throughout this work, let $\nu_p^X$ denote the $p$-mean of a random variable $X$ whenever defined. We use $\nu_p$ when the random variable in question is unambiguous.

	\begin{proposition}\label{bounds}
		Let $X$ be a random variable supported on the possibly infinite interval $(L,R)$. Then 
		$\nu_p\in(L,R)$ for all $p\in\D$.
	\end{proposition}

	\noindent The following fact was used implicitly in \cite{Kovchegov21}, but we prove it for completeness.

	\begin{proposition}\label{nup C1}
		The map $p\mapsto\nu_p$ is continuously differentiable on $\D\cap(1,\infty)$.
	\end{proposition}
	
	\noindent When investigating specific distribution families, we may assume that the scale and location parameters
	are 1 and 0 respectively unless otherwise noted. This is justified by the following, which implies that
	true positive skewness is preserved under positive affine transformations.
	
	\begin{proposition}\label{linearity}
		For any $c,s\in\R$ and $p\in\D_X$, $\nu_p^{cX+s} = c\nu_p^X+s$.
	\end{proposition}
	
	\noindent Next, we consider the asymptotic behavior of $\nu_p$ as $p\to\infty$. This requires $X$ to have finite
	moments of all orders, which clearly holds if $X$ has bounded support. We consider only continuous
	random variables, but analogous results hold in the discrete case.
	
	\begin{proposition}\label{nu->midpoint, bounded}
		Let $X$ be a continuous random variable supported on the finite interval $(L,R)$. Then 
		$\nu_p\to (L+R)/2$ as $p\to\infty$.
	\end{proposition}
		
	\noindent Suppose instead $X$ has infinite support that is bounded below. We have an analogous result if the 
	support is instead bounded from above.

	\begin{proposition}\label{nu->inf, halfline}
		Let $X$ be a continuous random variable supported on $(L,\infty)$ for finite $L$. If $X$ has 
		finite moments of all orders and $P(X>x)>0$ holds for every $x>L$, then $\nu_p\to\infty$ as
		$p\to\infty$.
	\end{proposition}

	\noindent A consequence of Proposition \ref{nu->inf, halfline} is that no distribution with support on the 
	positive half-line is truly negatively skewed. Similarly, no distribution with support on the negative
	half-line is truly positively skewed.

\subsection{Examples of true skewness}\label{Ss:Levy, Weibull, Skew}

True positive skewness has already been shown in \cite{Kovchegov21} for several distributions: exponential, gamma, beta (with the mode in the left half-interval), log-normal, and Pareto. Using Lemma \ref{Kovchegov L2}, we establish true positive skewness of L\'evy and chi-squared distributions, and identify the parameter regions for which Weibull and skew-normal distributions are truly skewed. Recall that L\'evy distribution has undefined Pearson moment coefficient of skewness because it has no finite integer moments. Yet, as we already mentioned in Remark \ref{rem:infmoment}, finiteness of integer moments is not required for true skewness. Thus, to the authors' knowledge, L\'evy distribution's positive skewness is formally established for the first time in this work.

	\begin{definition}
		The {\bf L\'evy distribution} with location parameter $\mu\in\R$ and scale parameter
		$\lambda>0$ is a continuous probability distribution, denoted by $\Levy(\mu,\lambda)$, with
		the density function
		\begin{equation}\label{Levy density}
			f(x;\mu,\lambda) := \sqrt{\frac{\lambda}{2\pi}}\,\frac{e^{-\frac{\lambda}{2(x-\mu)}}}
				{(x-\mu)^{3/2}}, \qquad x>\mu.
		\end{equation}
	\end{definition}
	
	\begin{definition}
		The {\bf chi-squared distribution} with $k\in\N$ degrees of freedom is a continuous probability 
		distribution, denoted by $\ChiSquared(k)$, with the density function
		\begin{equation}\label{chi-squared density}
			f(x;k) := \frac{x^{(k/2)-1}e^{-x/2}}{2^{k/2}\Gamma(k/2)}, \qquad x>0,
		\end{equation}
	\end{definition}
	
	\begin{definition}
		The {\bf Weibull distribution} with shape paramter $k>0$ and scale parameter $\lambda>0$ is a 
		continuous probability distribution, denoted by $\Weibull(k,\lambda)$, with the density 
		function
		\begin{equation}\label{Weibull density}
			f(x;k,\lambda) := \frac k\lambda \left(\frac x\lambda\right)^{k-1} e^{-(x/\lambda)^k},
			\qquad x>0.
		\end{equation}
	\end{definition}
	
	\begin{definition}
		The {\bf skew-normal distribution} with shape parameter $\alpha\in\R$ is a continuous
		probability distribution, denoted by $\SkewNormal(\alpha)$, with the density function
		\begin{equation}\label{skew-normal density}
			f(x;\alpha) = 2\phi(x)\Phi(\alpha x), \qquad x\in\R,
		\end{equation}
		where $\phi(x)=(2\pi)^{-1/2}e^{-x^2/2}$ and $\Phi(x)=\int_{-\infty}^x\phi(t)dt$ are respectively
		the density and distribution functions of the standard normal Gaussian distribution.
	\end{definition}

	\noindent A common strategy for showing the true skewness of the preceding distributions employs the following observation.

	\begin{lemma}\label{nup clopen}
		Suppose there exists a constant $C$ such that, for every $p\in\D\setminus\{0\}$, $\nu(\cdot)$ is 
		increasing at $p$ whenever $\nu(p)>C$. If $\nu(p')>C$ for some $p'\in\D\setminus\{0\}$, then 
		$\nu(\cdot)$ is increasing on $\D\cap[p',\infty)$.
	\end{lemma}

	\noindent When considering distribution families, we can always assume the location parameter is 0 and the scale parameter is 1, since they do not affect the direction of skewness (see Proposition \ref{linearity}).

	\begin{theorem}\label{Levy TPS}
		The $\Levy(\mu,\lambda)$ distribution is truly positively skewed.
	\end{theorem}
	
	\begin{theorem}\label{chi-squared TPS}
		The $\ChiSquared(k)$ distribution is truly positively skewed.
	\end{theorem}

	\begin{theorem}\label{Weibull TPS}
		The $\Weibull(k,\lambda)$ distribution is truly positively skewed if and only if
		$0<k<\frac1{1-\log2}$. Moreover, it is never truly negatively skewed.
	\end{theorem}
	
	\noindent The L\'evy, chi-squared, and Weibull distributions are supported only on the half-line, but similar techniques apply if the distribution is supported on the entire line.
	
	\begin{theorem}\label{Skew-normal TPS}
		The $\SkewNormal(\alpha)$ distribution is truly positively skewed if $\alpha>0$,
		truly negatively skewed if $\alpha<0$, and symmetric if $\alpha=0$.
	\end{theorem}

\subsection{True skewness under limits in distribution}\label{Ss:Weak Limits}

It is reasonable to conjecture that true skewness is preserved under uniform limits of distribution functions since Lemma \ref{Kovchegov L2} implies that true skewness is, essentially, a feature of a continuous random variable's density function. However, one must be somewhat careful: let $X_n\sim\Gam(n,\lambda)$ be a sequence of independent gamma random variables, for some fixed $\lambda$, such that each $X_n$ can be expressed as a sum of $n$ i.i.d. exponential random variables. We know from \cite{Kovchegov21} that the $X_n$'s are truly positively skewed, but the central limit theorem implies that their limit in distribution is Gaussian and thus symmetric.
	
Therefore, we introduce the notion of \emph{true non-negative skewness} to refer to a random variable whose $p$-means are non-decreasing, i.e., $d\nu_p/dp\ge0$, as opposed to the strict increase required by true positive skewness. Notice that truly positively skewed as well as symmetric distributions are truly non-negatively skewed.

	\begin{theorem}\label{weak limits}
		Let $F_n,F$ be the distribution functions of $X_n,X$ respectively. Suppose that $F_n\to F$ 
		uniformly and that $\sup_nE\big[|X_n|^{p+\epsilon}\big]<\infty$ holds for some $\epsilon>0$ and every $p\in\D_X$. If the
		$X_n$'s are truly non-negatively skewed, then $X$ is truly non-negatively skewed.
		
		If, moreover, there exists a constant $C>0$ such that ${d \over dp}\nu_p^{X_n}>C$ holds for every
		$n$ and every $p\in\D_X$, then $X$ is truly positively skewed.
	\end{theorem}
	
\noindent The condition that the distribution functions $F_n$ converge uniformly is satisfied often in practice, since uniform convergence follows from pointwise convergence if $F$ is continuous (see, e.g., \cite[Exercise 3.2.9] {Durrett2019}), which holds if $X$ has no point masses. Alternatively, uniform convergence of the distribution functions holds if the characteristic functions converge uniformly.
	
Theorem \ref{weak limits} can be used when considering the parameter regions of true skewness for certain distribution families. As an example, consider a sequence $X_n\sim\Weibull(k_n,1)$ of independent Weibull random variables, where $k_n\uparrow\frac{1}{1-\log 2}$; these are truly positively skewed by Theorem \ref{Weibull TPS}. It is clear that the distribution functions of the $X_n$'s converge uniformly to the distribution function of $X\sim\Weibull(\frac1{1-\log 2},1)$, and one can show directly that the $p$-th moments of the $X_n$'s are uniformly bounded for every $p\ge1$ (see, e.g., \cite[Eq.~2.63d]{Rinne08}). Theorem \ref{weak limits} then implies that $X$ is truly non-negatively skewed.

\subsection{Additional criteria for true skewness}\label{Ss:Additional Criteria}

	In this section, we establish two criteria for true positive skewness, one based upon a stochastic representation and the other based upon numerically verifiable conditions.

	\begin{theorem}\label{u(X)}
		Let $X$ be a continuous random variable with density function decreasing on its support. 
		If $u:\R\to\R$ is convex and strictly increasing on the support of $X$, then $u(X)$ is truly
		positively skewed.
	\end{theorem}
	
	\noindent One immediate application of this theorem is when $X\sim\Exp(\lambda)$ is exponentially distributed. It is well-known that $ke^X\sim\Pareto(k,\lambda)$ and that $X^2\sim\Weibull(\frac12,\frac1{\lambda^2})$, so Theorem \ref{u(X)} immediately yields the true positive skewness of the $\Pareto(k,\lambda)$ and $\Weibull(\frac12,\frac1{\lambda^2})$ distributions.

Our second criteria for true positive skewness does not rely on the $p$-means of a distribution other than its mode and median, provided that it has a density function supported on the half-line which is twice continuously differentiable. It also does not require knowledge of the density $f$ expressed in terms of elementary functions, which has conveniently been provided in each of the specific distributions previously examined; instead, it requires certain bounds on the logarithmic derivative within certain intervals. This theorem may have applications in numerically checking true positive skewness for one-sided stable distributions, for which very little descriptive information is known in general, given specific parameter values.

	\begin{theorem}\label{criteria}
		Let $X$ have support on $(0,\infty)$ with continuous unimodal density $f\in C^2(0,\infty)$.
		Suppose $f$ has exactly two positive inflection points $\theta_1,\theta_2$ such that 
		$\theta_1<\nu_0<\theta_2$, and
		\begin{enumerate}
			\item $f'/f>1/\nu_0$ on $(0,\theta_1)$,
			\item $f'/f>-1/\nu_0$ on $(\theta_2,\infty)$.
		\end{enumerate}
		If $\nu_1>(\nu_0+\theta_2)/2$, then $X$ is truly positively skewed.
	\end{theorem}

	\begin{corollary}\label{criteria 2}
		Let $X$ have support on $(0,\infty)$ with continuous unimodal density $f\in C^2(0,\infty)$.
		Suppose $f$ has exactly one positive inflection point $\theta>\nu_0$. If $\nu_1>(\nu_0+
		\theta)/2$, then $X$ is truly positively skewed.
	\end{corollary}

	\begin{remark}\label{criteria relaxations}
		The conditions of Theorem \ref{criteria} can 
		be relaxed, at the potential cost of practical ease. In particular, the following require 
		one to compute $c_1$, as defined in \eqref{cp}.
	
		\begin{enumerate}[label=(\alph*)]
			\item The condition $\nu_1>(\nu_0+\theta_2)/2$ serves to guarantee that $f$ is convex 
			at $\nu_p+c_p$ for all $p$, but the required convexity can certainly be achieved for 
			a weaker lower bound on $\nu_1$. Indeed, one can see in the proof in Section \ref{S:Proofs}
			that $\nu_1+c_1>\theta_2$ is sufficient; we obtain \eqref{nup+cp>theta} via Lemma 
			\ref{criteria l4}. Note that this is not \emph{always} a weaker lower bound.
			
			\item The quantity $\nu_0$ in conditions (1) and (2) can be replaced by 
			$\nu_1-c_1$.\footnote{\ This makes the proof significantly lengthier; in fact, Lemmas 
			\ref{criteria l2}, \ref{criteria l3}, and \ref{criteria l4} are otherwise unnecessary.
			For the proof of Theorem \ref{criteria} in Section \ref{S:Proofs}, we present the most 
			general argument.} As we show later, $\nu_1-c_1<\nu_0$, so actually this replacement
			creates a stronger condition on the lower bound of $f'/f$ to the left of the mode and a 
			weaker condition on the lower bound of $f'/f$ to the right of the mode. This replacement 
			is useful when the density has a steeper right tail.

			\item Similarly, condition (2) only needs to hold on $(\nu_1+c_1,\infty)$.
		\end{enumerate}
	\end{remark}
	
	\begin{example}[Log-logistic distribution]
		The log-logistic distribution with shape parameter $\beta>0$ has density function
		$$f(x;\beta) := \frac{\beta x^{\beta-1}}{(1+x^\beta)^2},\qquad x>0.$$
		If $0<\beta\le1$, then $f$ strictly decreasing and true positive skewness follows from 
		Proposition \ref{Kovchegov P2}. Suppose $\beta>1$. One can verify that $f$ is unimodal 
		with mode
		$$\nu_0=\left(\frac{\beta-1}{\beta+1}\right)^{1/\beta},$$
		median $\nu_1=1$, and inflection points
		$$\theta^\pm = \left(\frac{2\beta^2-2\pm\beta\sqrt{3\beta^2-3}}
			{\beta^2+3\beta+2}\right)^{1/\beta}.$$
		Straightforward computations show that $\theta^-\le0$ holds if and only if $\beta\le2$, and 
		$\theta^+>\nu_0$ holds if and only if $\beta>1$. Moreover, $\theta^+<1$ holds if and only if 
		$1\le\beta<2$. Therefore, the log-logistic distribution is truly positively skewed if
		$0<\beta<2$, by Corollary \ref{criteria 2}.
	\end{example}


\section{Proofs}\label{S:Proofs}

\subsection{Proofs of results for Section \ref{S:Introduction}}

	\begin{proof}[Proof of Proposition \ref{relaxed Kovchegov P2}]
		Clearly $L$ is finite, otherwise $f$ could not be a non-increasing density function. Notice 
		that $\nu_p<(L+R)/2$ for all $p\in\D$. Otherwise \eqref{nup def 2 integral} fails to hold 
		since $f(\nu_p+y)\le f(\nu_p-y)$ and $R-\nu_p<\nu_p-L$ by assumption.
		
		Now the existence of $y_1<y_2$ such that $f(y_1)>f(y_2)$ implies that there exists a non-singleton 
		interval in $(0,\nu_p-L)$ on which $f(\nu_p+y)<f(\nu_p-y)$. Then strict stochastic dominance of 
		$\frac1{H_p}y^{p-1}f(\nu_p+y)\1_{(0,R-\nu_p)}(y)$ over $\frac1{H_p}y^{p-1}f(\nu_p-y)\1_{(0,\nu_p-L)}
		(y)$ follows by integrating each density, applying monotonicity of the integral, and using equation 
		\eqref{nup def 2 integral}.
	\end{proof}

\subsection{Proofs of results for Section \ref{Ss:Properties}}

	\begin{proof}[Proof of Proposition \ref{bounds}]
		If $\nu_p\le L$, then $E[(\nu_p-X)_+^{p-1}]=0$ while $E[(X-\nu_p)_+^{p-1}]>0$, contradicting
		\eqref{nup def 2}. Analogous if $\nu_p\ge R$.
	\end{proof}
	
	\begin{proof}[Proof of Proposition \ref{nup C1}]
		Consider the function $\Phi:\R\times\D\cap(1,\infty)\to\R$ given by 
		\begin{equation*}
			\Phi(a,p) := E[(X-a)_+^{p-1}]-E[(a-X)_+^{p-1}]
		\end{equation*}
		Both integrands $(X-a)_+^{p-1}$ and $(a-X)_+^{p-1}$ are continuously differentiable functions 
		of $a$ and $p$ within their support and are dominated by an integrable function since
		 $E|X|^{p-1}<\infty$. By the Leibniz integral rule, observe that
		\begin{equation*}
			\frac{\partial\Phi}{\partial a}(a,p) = -(p-1)\Big(E[(X-a)_+^{p-2}]+E[(a-X)_+^{p-2}]\Big),
		\end{equation*}
		which is strictly negative and finite for all $a\in\R$ and $p\in\D\cap(1,\infty)$. The map 
		$p\mapsto (\nu_p,p)$ is the zero level curve of $\Phi$ and so $p\mapsto\nu_p$ is continuously 
		differentiable by the implicit function theorem.
	\end{proof}

	\begin{proof}[Proof of Proposition \ref{linearity}]
		Let $\D:=\D_{cX+s}$. It is easy to see that $\D_X=\D$. For every $p\in\D$,
		\begin{align*}
			E\left[\Big((cX+s)-(c\nu_p^X+s)\Big)_+^{p-1}\right] &= c^{p-1}E[(X-\nu_p^X)_+^{p-1}] \\
			&= c^{p-1}E[(\nu_p^X-X)_+^{p-1}] = E\left[\Big((c\nu_p^X+s)-(cX+s)\Big)_+^{p-1}\right]
		\end{align*}
		holds, and the result follows.
	\end{proof}
	
	\begin{proof}[Proof of Proposition \ref{nu->midpoint, bounded}]
		By Proposition \ref{linearity}, it suffices to show that if $X$ has support on $(0,1)$, 
		then $\nu_p\to1/2$.
		
		Let $\epsilon\in(0,1/2)$ be given. Suppose for the sake of contradiction that there exists a 
		subsequence $\nu_{p_k}$ such that $\nu_{p_k}<1/2-\epsilon$ for all 
		$k\in\N$. For $p_k>1$, we have
		\begin{align}\label{midpoint rbound}
			E[(X-\nu_{p_k})^{p_k-1}\1\{X>\nu_{p_k}\}]
			&> E[(X-1/2+\epsilon)^{p_k-1}\1\{X>1-\epsilon\}] \nonumber \\
			&> E[2^{-(p_k-1)}\1\{X>1-\epsilon\}] \nonumber \\
			&= 2^{-(p_k-1)}P(X>1-\epsilon)
		\end{align}
		and
		\begin{align}\label{midpoint lbound}
			E[(\nu_{p_k}-X)^{p_k-1}\1\{X<\nu_{p_k}\}]
			&< E[(1/2-\epsilon)^{p_k-1}\1\{X<1/2-\epsilon\}] \nonumber \\
			&= (1/2-\epsilon)^{p_k-1}P(X<1/2-\epsilon).
		\end{align}
		From \eqref{midpoint rbound} and \eqref{midpoint lbound}, we have
		$$\frac{E[(\nu_{p_k}-X)^{p_k-1}\1\{X<\nu_{p_k}\}]}{E[(X-\nu_{p_k})^{p_k-1}\1\{X>\nu_{p_k}\}]}
			< \frac{P(X<1/2-\epsilon)}{P(X>1-\epsilon)} \cdot (1-2\epsilon)^{p_k-1} \to 0$$
		as $k\to\infty$. This contradicts \eqref{nup def 2}. Thus, every subsequence $\nu_{p_k}$ has 
		only finitely many points that lie below $1/2-\epsilon$. We can similarly show that every 
		subsequence has only finitely many points that lie above $1/2+\epsilon$. It follows by taking 
		$\epsilon\downarrow0$ that every subsequence converges to $1/2$, which completes the proof.
	\end{proof}
	
	\begin{proof}[Proof of Proposition \ref{nu->inf, halfline}]
		We may assume $L=0$. Suppose for the sake of contradiction that there exists $M>0$ such that 
		$\nu_p<M$ for all $p\ge1$. By Proposition \ref{linearity}, we may assume $M=1/2$. Note that 
		$\nu_p-X<1/2$ on $\{X<\nu_p\}$, so by the bounded convergence theorem $E[(\nu_p-X)_+^{p-1}]
		\to0$ as $p\to\infty$. On the other hand, 
		$$E[(X-\nu_p)_+^{p-1}] \ge E[(X-\nu_p)^{p-1}\1\{X>\nu_p+1\}] \ge P(X>3/2) > 0,$$
		hence \eqref{nup def 2} fails to hold for large $p$, contradiction.
	\end{proof}

\subsection{Proofs of results for Section \ref{Ss:Levy, Weibull, Skew}}

	\begin{proof}[Proof of Lemma \ref{nup clopen}]
		Suppose for the sake of contradiction that the set 
		$$A:=\{p\in\D\cap[p',\infty):\nu_p\le C\}$$
		is non-empty. Since $p\mapsto\nu_p$ is continuous by Proposition \ref{nup C1}, then $A$ is closed
		and thus contains its infimum $q>p'$. By construction, $\nu_p>C$ holds for every $p\in[p',q)$, 
		hence $\nu(\cdot)$ is increasing on $[p',q)$ and so $\nu_q\ge\nu_{p'}>C$, contradicting the fact 
		that $q\in A$.
	\end{proof}

	\begin{proof}[Proof of Theorem \ref{Levy TPS}]
		By Proposition \ref{linearity}, it suffices to consider the $\Levy(0,1)$ distribution. Let
		$f$ be its density function, i.e., $f(x)\equiv f(x;0,1)$ as in \eqref{Levy density}.
		
		Fix $p\in\D\setminus\{0\}$. To show that $\nu(\cdot)$ is increasing at $p$, it suffices to 
		show that the log density ratio of the left and right parts
		$$R_p(x) := \log\left(\frac{f(\nu_p-x)}{f(\nu_p+x)}\right) 
			= \frac32(\log(\nu_p + x)-\log(\nu_p - x)) - \frac1{2(\nu_p - x)} + \frac1{2(\nu_p + x)},$$
		for $x\in[0,\nu_p)$, has exactly one positive critical point.
		Indeed, observe that $R_p(x)\to-\infty$ as $x\to\nu_p-$ and that $R_p(0)=0$. Moreover, 
		\eqref{nup def 2 integral} implies that $R_p(\cdot)$ cannot be non-positive for all $x>0$. 
		Therefore, if $R_p(\cdot)$ has a single positive critical point, then it must be a maximum, 
		and it follows that the conditions of Lemma \ref{Kovchegov L2} are satisfied and so 
		$\nu(\cdot)$ is increasing at $p$.
		
		To identify the positive critical points, observe that
		$$R'_p(x) = \frac3{2(\nu_p - x)} + \frac3{2(\nu_p + x)} - \frac1{2(\nu_p - x)^2} 
			- \frac1{2(\nu_p + x)^2},$$
		so the critical points are solutions to the equation
		$$3(\nu_p-x)(\nu_p+x)^2 + 3(\nu_p+x)(\nu_p-x)^2 - (\nu_p+x)^2 - (\nu_p - x)^2 = 0.$$
		Further simplification yields a quadratic equation in $x$,
		$$(3\nu_p+2)x^2 + \nu_p^2(3\nu_p-2) = 0$$
		which has exactly one root in $(0,\nu_p)$ if $\nu_p>2/3$.
		This holds for arbitrary $p\in\D$, so we may conclude
		true positive skewness if $\nu_1>2/3$ and $\nu_1>\nu_0$ hold, by Lemma \ref{nup clopen}. 
		The mode and median of the L\'evy distribution are well-known (see, e.g., \cite{Nolan20}) 
		and can be computed directly as $\nu_0=1/3$ and $\nu_1=\frac12(1-\text{erfc}(\frac12))^{-2}
		\approx 2.17$, which completes the proof.
	\end{proof}
	
	\begin{proof}[Proof of Theorem \ref{chi-squared TPS}]
    	Fix $k\in\N$ and let $f$ be the density function of $\chi^2(k)$, i.e., $f(x)\equiv f(x;k)$ as in \eqref{chi-squared density}. 
    	Fix $p\in\D\setminus\{0\}$. There are two cases.
    
   	\textbf{Case 1:} $k=1,2$. Here, $f(x)$ is decreasing on its domain and so true positive skewness follows from 
    	Proposition \ref{Kovchegov P2} and Theorem \ref{Kovchegov T1}.
    
   	\textbf{Case 2:} $k>2$. As in the proof of Theorem \ref{Levy TPS}, to show that $\nu(\cdot)$ is increasing at $p$, it 
    	suffices to show that the ratio of the left and right parts
    		$$R_p(x) := \log\left(\frac{f(\nu_p-x)}{f(\nu_p+x)}\right) 
    		= \left(\frac k2 - 1\right) \log\left(\frac{\nu_p-x}{\nu_p+x}\right)+x$$
	for $x\in[0,\nu_p)$, has exactly one positive critical point, since $R_p(x)\to-\infty$ as $x\to\nu_p-$ and $R_p(0)=0$. 
	Observe that
		$$R'_p(x) = \frac{(2-k)\nu_p}{\nu_p^2-x^2}+1,$$
	so the critical points are solutions to the equation $x^2 - \nu_p^2 + (k-2)\nu_p = 0$, i.e., 
		$$x = \pm\sqrt{\nu_p^2-(k-2)\nu_p}.$$
	There is exactly one positive critical point when $\nu_p^2-(k-2)\nu_p>0$. Since $f$ has support on the positive half 
	line, then $\nu_p$ is always non-negative by Proposition \ref{bounds}. Thus $\nu(\cdot)$ is increasing at $p$ if 
	$\nu_p-k+2>0$. Now by Lemma \ref{nup clopen}, true positive skewness of $\chi^2(k)$ follows if the median satisfies 
	$\nu_1>k-2$, but this inequality is well-known (see, e.g., Sen \cite[Eq.~(1.4)]{Sen89} and the references therein).
	\end{proof}
	
	\begin{proof}[Proof of Theorem \ref{Weibull TPS}]
		By Proposition \ref{linearity}, it suffices to consider the $\Weibull(k,1)$ distribution.
		Let $f$ be its density function, i.e., $f(x)\equiv f(x;k,1)$ as in \eqref{Weibull density}.
		
		It is well-known (see, e.g., \cite{Rinne08}) that the $\Weibull(k,1)$ distribution
		has finite moments of all orders and is unimodal with median and mode given by
		\begin{equation}\label{Weibull mode/median}
			\nu_0 = \left(\frac{k-1}{k}\right)^{1/k},\qquad \nu_1 = (\log 2)^{1/k}.
		\end{equation}
		Notice that $\nu_0<\nu_1$ holds if and only if $k<(1-\log2)^{-1}$.
		
		Since the Weibull distribution has support on the positive half-line, then the second 
		part of the theorem follows from Proposition \ref{nu->inf, halfline}. For the first part 
		of the theorem, the ``only if'' part follows from \eqref{Weibull mode/median}.

		It remains to show the ``if'' part. For $k\le1$, $f$ is strictly decreasing, so we are
		done by Proposition \ref{Kovchegov P2}.
		
		Suppose $k>1$ and fix $p\ge1$. As in the proof of Theorem \ref{Levy TPS}, to show that 
		$\nu(\cdot)$ is increasing at $p$, it suffices to show that the log density ratio of the 
		left and right parts
		$$R_p(x):=\log\left(\frac{f(\nu_p-x)}{f(\nu_p+x)}\right)
			= (k-1)\left[\log(\nu_p-x)-\log(\nu_p+x)\right]-(\nu_p-x)^k+(\nu_p+x)^k$$
		for $x\in[0,\nu_p)$, has exactly one positive critical point, since $R_p(x)\to-\infty$ as
		$x\to\nu_p-$ and $R_p(0)=0$. Observe that
		$$R'_p(x) = k(\nu_p-x)^{k-1}+k(\nu_p+x)^{k-1}-\frac{k-1}{\nu_p-x}-\frac{k-1}{\nu_p+x},$$
		so the critical points are solutions to the equation
		\begin{equation}\label{Weibull Rp*}
			k(\nu_p^2-x^2)\left[(\nu_p-x)^{k-1}+(\nu_p+x)^{k-1}\right]-2(k-1)\nu_p = 0.
		\end{equation}
		Since $x\in[0,\nu_p)$, the binomial series for $(\nu_p-x)^{k-1}$ and 
		$(\nu_p+x)^{k-1}$ converge, hence
		$$(\nu_p-x)^{k-1}+(\nu_p+x)^{k-1} = 2\sum_{n=0}^\infty \binom{k-1}{2n} \nu_p^{k-2n-1} x^{2n},$$
		in the notation of the generalized binomial coefficient. Substituting into \eqref{Weibull Rp*}
		yields the equation
		\begin{equation}\label{Weibull E1}
			g(x) := 2k\sum_{n=1}^\infty \left(\binom{k-1}{2n} - \binom{k-1}{2n-2}\right) 
				\nu_p^{k-2n+1} x^{2n} + 2\nu_p(k\nu_p^k-k+1) = 0,
		\end{equation}
		where
		\begin{equation}\label{Weibull E2}
			\binom{k-1}{2n}-\binom{k-1}{2n-2} = \frac{k(k-1)\dots(k-2n+2)(k+1-4n)}{(2n)!}.
		\end{equation}
		
		We analyze the sign changes of the coefficients in \eqref{Weibull E1} to determine the number 
		of its positive roots by splitting into several cases for the value of $k$.
		
		\textbf{Case 1:} $2<k<3$. Then $k,k-1,k-2$ are positive and $k-3,\dots,k-2n+2$ are negative. 
		There are an even number (possibly zero, if $n=1,2$) of negative factors in $k(k-1)\dots(k-2n+2)$,
		so it is positive. Also note that $k+1-4n<0$ for all $n\ge1$, hence \eqref{Weibull E2} is 
		negative for all $n\ge1$.
		
		For the series expression $g(x)=\sum_{m=0}^\infty a_mx^m$, we have shown that $a_m<0$ holds for 
		non-zero even $m$ and $a_m=0$ holds for odd $m$. By Descartes' rule of signs for infinite series,
		$g^*_p$ has no positive real roots if $a_0\le0$ and has at most one positive real root if $a_0>0$.
		
		Suppose $a_0\le0$ such that $g$ has no positive real roots. By extension, $R'_p$ has no positive
		real roots, so $R_p$ has no positive extrema and is strictly monotonic on $[0,\nu_p)$. Since 
		$g_p(0)=0$ and $\lim_{x\to\nu_p-}g_p(x)=-\infty$, then $g_p$ is strictly negative on $[0,\nu_p)$,
		hence $f(\nu_p-x)<f(\nu_p+x)$ for all $x\in(0,\nu_p)$. By monotonicity of the integral, this
		contradicts \eqref{nup def 2 integral}. Thus $a_0>0$ holds, and $g$ has exactly one positive root
		in $(0,\nu_p)$, which implies that $\nu(\cdot)$ is increasing at $p$.
		
		This holds for
		arbitrary $p\ge1$, so $\nu(\cdot)$ is increasing on $\D\setminus\{0\}$. Moreover, since $a_0=
		2\nu_p(k\nu_p^k-k+1)>0$ holds, then $\nu_p>\left(\frac{k-1}{k}\right)^{1/k}=\nu_0$ holds for every 
		$p\ge1$,  and we are done.
		
		\textbf{Case 2:} $1<k<2$. The inequality $k+1-4n<0$ still holds for all $n\ge1$, but we now have 
		$k,k-1>0$ and $k-2,\dots,k-2n+2<0$. If $n>1$, then there are an odd number
		of negative factors in $k(k-1)\dots(k-2n+2)$, so the numerator of \eqref{Weibull E2} is positive. 
		If $n=1$, then the numerator is simply $k(k-3)<0$.
		
		Thus our coefficients in $g(x)=\sum_{m=0}^\infty a_mx^m$ are positive for even $m>2$ and zero for 
		odd $m$, with $a_2<0$. By Descartes' rule of signs, $g$ has at most two positive real roots if
		$a_0>0$ and at most one if $a_0\le0$.
		
		Suppose $a_0\le0$. If $g$ has a single positive root, then $\lim_{x\to\nu_p-}g(x)>0$. But 
		by continuity this limit tends to $-2(k-1)\nu_p<0$. If instead $g^*_p$ has no positive real roots,
		then by the same argument above we contradict \eqref{nup def 2 integral}. Thus $a_0>0$. In this 
		case, if $g$ has zero or two positive roots, then again $\lim_{x\to\nu_p-}g(x)>0$ and we reach a 
		contradiction. Hence $g$ has exactly one positive real root. We may now conclude in the same
		fashion as in Case 1.
		
		\textbf{Case 3:} $3<k<\frac1{1-\log2}$. We again look at the numerator of \eqref{Weibull E2}. If
		$n=1$, the numerator is $k(k-3)>0$. If $n=2$, the numerator is $k(k-1)(k-2)(k-7)<0$. If $n>2$, then
		$k(k-1)\dots(k-2n+2)$ has positive factors $k,\dots,k-3$, and an odd number of negative factors 
		$k-4,\dots,k-2n+2$. Additionally, $k+1-4n<0$ when $n>2$, so the numerator of \eqref{Weibull E2} is 
		positive.
		
		Our coefficients in $g(x)=\sum_{m=0}^\infty a_mx^m$ are zero for odd $m$, positive for even
		$m\ge6$, negative for $m=4$, and positive for $m=2$. This yields two sign changes and hence at most
		two positive roots of $g$ if $a_0>0$. Recall from Case 1 that $a_0>0$ holds if and only if 
		$\nu_p>\nu_0$ holds. If the latter holds, then $g$ must have exactly one positive root since $g$ 
		is negative at the right limit of its domain, and so $\nu(\cdot)$ is increasing at $p$. Now by 
		Lemma \ref{nup clopen}, true positive skewness follows if $\nu_1>\nu_0$, but this holds 
		immediately from \eqref{Weibull mode/median} and our assumption $k<\frac1{1-\log2}$.
				
		\textbf{Case 4:} $k=2$. We can plug $k=2$ directly into \eqref{Weibull Rp*} to obtain the roots
		$$x^2 = \nu_p^2-\frac12 = \nu_p^2-\nu_0^2.$$
		The argument in Case 1 can be adapted to show that $\nu_p>\nu_0$ holds for every $p\ge1$, 
		hence exactly one positive root exists.
		
		\textbf{Case 5:} $k=3$. Similarly, we plug $k=3$ into \eqref{Weibull Rp*} and obtain the roots
		$$x^4 = \nu_p^4-\frac23\nu_p = \nu_p(\nu_p^3-\nu_0^3)$$
		Again, since $\nu_p>\nu_0$ holds for every $p\ge1$, then exactly one positive root exists.
	\end{proof}
	
	\begin{proof}[Proof of Theorem \ref{Skew-normal TPS}]		
		Recall that $\phi,\Phi$ are the density and distribution functions of the standard Gaussian
		distribution. For simplicity, set
    	$$\Psi(x) := \sqrt{2\pi}\,\Phi(x) = \int_{-\infty}^xe^{-t^2/2}\,dt.$$
		We show true positive skewness for positive shape parameter $\alpha>0$; the proof for true 
		negative skewness for $\alpha<0$ is analogous, and clearly the skew-normal is simply the normal 
		distribution, which is symmetric, when $\alpha=0$. Thus, fix $\alpha>0$ and let $f$ be the
		density function of $\SkewNormal(\alpha)$, i.e., $f(x)\equiv f(x;\alpha)$ as in
		\eqref{skew-normal density}.
		
		Fix $p\in\D\setminus\{0\}$. To show that $\nu(\cdot)$ is increasing at $p$, it suffices to show 
		that the log density ratio of the left and right parts
		$$R_p(x) := \log\left(\frac{f(\nu_p-x)}{f(\nu_p+x)}\right)
    		= 2\nu_px + \log\Psi(\alpha(\nu_p-x)) - \log\Psi(\alpha(\nu_p+x))$$
		has exactly one positive root $x_0$, satisfying $R_p(x)>0$ for $0<x<x_0$ and $R_p(x)<0$ for 
		$x>x_0$. This condition holds if the following four conditions can be verified:
		\begin{enumerate}[leftmargin=4em,label=(\Roman*)]
			\item $R_p(0) = 0$;
			\item $R''_p(x)<0$ for all $x>0$.
			\item $R'_p(0)>0$; and
			\item $R'_p(x)<0$ for some $x>0$.
		\end{enumerate}
		
		\noindent Condition (I) holds trivially. To prove condition (II), define
		$$\psi(x) := \log\Psi(x),$$
		hence
		\begin{align}
			R_p(x) &= 2\nu_px + \psi(\alpha(\nu_p-x)) - \psi(\alpha(\nu_p+x)); \nonumber\\
			R_p'(x) &= 2\nu_p - \alpha\psi'(\alpha(\nu_p-x)) - \alpha\psi'(\alpha(\nu_p+x)); 
				\label{Rp' psi} \\
			R_p''(x) &= \alpha^2\psi''(\alpha(\nu_p-x)) - \alpha^2\psi''(\alpha(\nu_p+x)). 
				\nonumber
		\end{align}
		Condition (II) holds if and only if $\psi''(\alpha(\nu_p-x)) < \psi''(\alpha(\nu_p+x))$ holds 
		for every $x>0$. It suffices then to show that $\psi''$ is monotonically increasing everywhere, 
		i.e., that $\psi'''(x)>0$ holds for every $x\in\R$.
	
		The third logarithmic derivative of $\Psi$ is given by
		\begin{align}\label{skew-normal 1}
			\psi'''(x) &= \frac{\Psi'''(x)}{\Psi(x)} - \frac{3\Psi'(x)\Psi''(x)}{\Psi(x)^2} 
				+ 2\left(\frac{\Psi'(x)}{\Psi(x)}\right)^3 \nonumber \\
			&= \frac{x^2-1}{\Psi(x)}\,e^{-x^2/2} + \frac{3x}{\Psi(x)^2}\,e^{-x^2} 
				+ \frac{2}{\Psi(x)^3}\,e^{-3x^2/2} \nonumber \\
			&= \frac{e^{-3x^2/2}}{\Psi(x)^3}\Big[(x^2-1)e^{x^2}\Psi(x)^2 + 3xe^{x^2/2}\Psi(x) + 2\Big].
		\end{align}
		One may compute directly
		\begin{equation*}\label{skew-normal 2}\begin{split}
			\psi'''(1) &= \frac{e^{-3/2}}{\Psi(1)^3}\Big[3e^{1/2}\Psi(1) + 2\Big] > 0; \\
			\psi'''(-1) &= \frac{e^{-3/2}}{\Psi(-1)^3}\Big[-3e^{1/2}\Psi(-1)+2\Big] \approx 0.12 > 0,
		\end{split}\end{equation*}
		so if $\psi'''$ has no zeroes, then continuity implies that $\psi'''$ is positive everywhere. 
		
		Define $u\equiv u(x):=(\psi'(x))^{-1}=e^{x^2/2}\Psi(x)$, noticing that $u$ is strictly 
		positive. By \eqref{skew-normal 1}, $\psi'''(x)=0$ holds if and only if
		$$(x^2-1)u^2 + 3xu + 2 = 0,$$
		which, by the quadratic formula, holds exactly when
		\begin{equation*}
			u = \frac{-3x \pm \sqrt{x^2+8}}{2(x^2-1)} = \mp\frac{4}{\sqrt{x^2+8}\pm 3x} =: Q^\pm(x).
		\end{equation*}
		In other words, the zeroes of $h'''$ are solutions to the equations 
		\begin{equation}\label{u Q}
			u(x) = Q^\pm(x),
		\end{equation}
		with special care needed for the asymptotic values $x=\pm1$ for $Q^\pm$. We handled these 
		cases in \eqref{skew-normal 1}, so it suffices to consider $x\neq\pm1$.
	
		Observe that $Q^\pm(x)<0<u(x)$ holds if $x>1$, so solutions to \eqref{u Q} must fall in 
		$(-\infty,1)$. If $x<-1$, then both $Q^+(x)$ and $Q^-(x)$ are positive and, one can easily 
		verify the inequalities $0<Q^-(x)<Q^+(x)$. On the other hand, if $-1<x<1$, then $Q^+(x)<0<u(x)$
		whereas $Q^-(x)>0$. Therefore, it suffices to show that $u(x) < Q^-(x)$ for every $x<1$, 
		$x\neq-1$.
	
		The key observation is that $u(x)$ is simply the reflection of the Mills' ratio of the standard 
		Gaussian, i.e.,
		$$u(-x) = M(x) := \frac{1-\Phi(x)}{\phi(x)} = e^{x^2/2}\int_x^\infty e^{-t^2/2}\,dt,$$
		for which very precise bounds in terms of elementary functions exist. In particular, we use 
		the known bound (see, e.g., \cite[Eq.~7.8.4]{NIST:DLMF}, or \cite{Sampford, Shenton})
		$$M(x) < \frac{4}{\sqrt{x^2+8}+3x},\qquad \forall\,x>-1,$$
		which yields
		$$u(x) < \frac{4}{\sqrt{x^2+8}-3x} = Q^-(x),\qquad \forall\,x<1,\,x\neq-1.$$
		We have now shown that \eqref{u Q} has no solutions and thus that $\psi'''$ has no zeroes, 
		proving condition (II).
		
		Since $R_p(\cdot)$ cannot be non-positive on the entire half-line, as a consequence of 
		\eqref{nup def 2 integral}, then condition (II) implies condition (III). For condition (IV),
		observe that $\psi'(x)=e^{-x^2/2}\,\Psi(x)^{-1}$ is decreasing and tends to infinity as 
		$x\to-\infty$, so \eqref{Rp' psi} implies that $R'_p(x)\to-\infty$ as $x\to\infty$. 
		
		It remains to show that $\nu_p>\nu_0$ for every $p\in\D\setminus\{0\}$. Notice that 
		$R'_p(0)=2\nu_p-2\alpha\psi'(\alpha\nu_p)$ is increasing as a function of $\nu_p$ given 
		the decrease of $\psi'$. The mode $\nu_0$ is the unique value satisfying $f'(\nu_0)=0$, which 
		is equivalent to the equality $R'_0(0)=2\nu_0-2\alpha\psi'(\alpha\nu_0)=0$. Condition (III) 
		and the increase of $\nu_p\mapsto R'_p(0)$ now yield $\nu_0<\nu_p$ for every 
		$p\in\D\setminus\{0\}$.
	\end{proof}

\subsection{Proofs of results for Section \ref{Ss:Weak Limits}}

	\begin{proof}[Proof of Theorem \ref{weak limits}]
	 	To prove the first part of the theorem, fix $p\in\D_X$ and define
		$$f_n(a):=E|X_n-a|^p,\quad a\in\R; \quad\quad\quad \nu_n:=\argmin_{a\in\R}f_n(a);$$
		$$f(a):=E|X-a|^p,\quad a\in\R; \quad\quad\quad \nu:=\argmin_{a\in\R}f(a);$$
		so that $\nu_n$ and $\nu$ are the $p$-means of $X_n$ and $X$ respectively by Definition 
		\ref{nup def 1}.
		
		Let $a_n$ be any sequence of numbers converging to $a$.
		Observe that the random variable $X_n-a_n$ has distribution function $\tilde F_n(x)=F_n(x+a_n)$.
		Similarly, $X-a$ has distribution function $\tilde F(x)=F(x+a)$, so the uniform convergence 
		$F_n\to F$ implies the pointwise convergence $\tilde F_n(x)\to\tilde F(x)$. Therefore, $X_n-a_n$
		converges to $X-a$ in distribution.
		
		By Jensen's inequality, we have
		$$\sup_nE|X_n-a_n|^{p+\epsilon} \le 2^{p+\epsilon}\sup_n(E|X_n|^{p+\epsilon}
			+ |a_n|^{p+\epsilon})<\infty,$$
		where finiteness holds by the convergence of $a_n\to a$ and by assumption on the $X_n$'s. 
		It follows from, e.g., \cite[Exercise 3.2.5]{Durrett2019}, that
		\begin{equation}\label{weak limits 1}
			f_n(a_n)=E|X_n-a_n|^p\to E|X-a|^p=f(a).
		\end{equation}
				
		Suppose that the $\nu_n$'s are contained in a compact interval. Then every subsequence 
		$\nu_{n_k}$ has a limit point $\nu_*$. Since $f_{n_k}(\nu_{n_k})\le f_{n_k}(a)$ for all $a\in\R$,
		then by taking $k\to\infty$ we obtain $f(\nu_*)\le f(a)$ holds for every $a\in\R$ via 
		\eqref{weak limits 1}. Then $\nu_*$ minimizes $f$, and since the minimizer of $f$ is unique 
		(see Definition \ref{nup def 1}), we obtain $\nu_*=\nu$. Every subsequence of $\nu_n$ converges 
		to $\nu$, hence $\nu_n\to\nu$.
		
		Consider $\nu_n$ and $\nu$ as functions of $p$, so $\nu_n$ converges pointwise to $\nu$. True 
		non-negative skewness of $X_n$ implies $\nu_n$ is non-decreasing, and the pointwise limit of 
		monotone functions is monotone, hence $\nu$ is non-decreasing.
		
		It remains to show that the $\nu_n$'s are contained in a compact interval. Suppose otherwise,
		so we have a subsequence $\nu_{n_k}\to\infty$ as $k\to\infty$. (The argument is similar if 
		$\nu_{n_k}\to-\infty$.) We have from \eqref{weak limits 1} the pointwise convergence $f_n\to f$, 
		so for fixed $\delta>0$, we may choose large enough $K$ such that $\nu_{n_k}>\nu$ and $|f_{n_k}
		(\nu)-f(\nu)|<\delta$ for all $k\ge K$. Since $f_{n_k}$ is strictly convex with minimizer 
		$\nu_{n_k}$, then $f_{n_k}(a)<f(\nu)+\delta$ for all $a\in[\nu,\nu_{n_k}]$.
		
		Clearly $f\to\infty$ as $a\to\infty$, so choose $x>\nu_{n_K}>\nu$ large enough such that 
		$f(x)>f(\nu)+2\delta$. For any $N>0$, there exists $k>K$ large enough such that 
		$n_k>N$ and $\nu_{n_k}>x$ such that $f_{n_k}(x)<f(\nu)+\delta<f(x)-\delta$ as shown above. 
		This contradicts the pointwise convergence $f_{n_k}(x)\to f(x)$, and so concludes the proof of 
		the first part.
		
		The second part of the theorem now easily follows from the fact that $\nu_n(p)\to\nu_n(p)$. 
		Indeed, for $p>p'$, we have by assumption that $\nu_n(p)-\nu_n(p')>C(p-p')$ holds for every 
		$n\in\N$, so taking limits yields $\nu(p)-\nu(p')\ge C(p-p')>0$, hence $\nu(\cdot)$ is strictly
		increasing.
	\end{proof}

\subsection{Proofs of results for Section \ref{Ss:Additional Criteria}}
	
	\begin{proof}[Proof of Theorem \ref{u(X)}]
		Let $X$ be a continuous random variable with density function $f_X$ decreasing on its support $(0, R)$ for possibly infinite $R$. Let $Y=u(X)$ where $u$ is convex and strictly increasing on the support of $X$. We define $w=u^{-1}$. Note that $\D_X=\D_Y=[1,\infty)$. We write $\D$ for this domain. Let $f_Y$ be the density of $Y$. It suffices to show that
		\begin{equation}\label{oneroot}
			\log\left(\frac{f_Y(\nu_p - y)}{f_Y(\nu_p + y)}\right)
		\end{equation}
		changes sign exactly once on the interval $(0, \infty)$, for all $p\in\D$. First note that as 
		$y$ increases, $\nu_p-y$ approaches 0. Once $y>\nu_p$, $f_Y(\nu_p - y)= 0$. Thus, on the 
		interval $y\in (\nu_p, \infty)$,
		\begin{equation*}
			\log\left(\frac{f_Y(\nu_p - y)}{f_Y(\nu_p + y)}\right) = -\infty
		\end{equation*}
		On the interval $u\in (0, \nu_p]$, we have $\{\nu_p - y\} \in [0, \nu_p)$, and $\{\nu_p + y\} 
		\in [\nu_p, 2\nu_p)$. We expand $f_Y$:
		\begin{align}
			\log\left(\frac{f_Y(\nu_p - y)}{f_Y(\nu_p + y)}\right) 
			&= \log\left(\frac{-f_X[w(\nu_p - y)]w'(\nu_p - y)}{-f_X[w(\nu_p + y)]w'(\nu_p + y)}\right)
				\nonumber\\
			&=\log\left(\frac{f_X[w(\nu_p - y)]}{f_X[w(\nu_p + y)]}\right) + 
				\log\left(\frac{w'(\nu_p - y)}{w'(\nu_p + y)}\right)\label{LHSRHS}
		\end{align}
		Because $w=u^{-1}$ is an increasing function and $f_X$ is a decreasing function, 
		\begin{equation*}
			f_X[w(\nu_p - y)] > f_X[w(\nu_p + y)]
		\end{equation*}
		\begin{equation*}
			\log\left(\frac{f_X[w(\nu_p - y)]}{f_X[w(\nu_p + y)]}\right) > 0
		\end{equation*}
		The function $u$ is convex and increasing, $w$ is concave, and therefore $w'$ is decreasing. 
		It follows that 
		\begin{equation*}
			w'(\nu_p - y) > w'(\nu_p + y) 
		\end{equation*}
		\begin{equation*}
			\log\left(\frac{w'(\nu_p - y)}{w'(\nu_p + y)}\right) > 0
		\end{equation*}
		which implies that \eqref{LHSRHS} is strictly positive on $(0, \nu_p]$. Since \eqref{oneroot} 
		is strictly positive on $(0, \nu_p]$ and strictly negative on $(\nu_p, \infty)$, it changes 
		sign exactly once on the interval $(0, \infty)$. This satisfies Lemma \ref{Kovchegov L2} for 
		all $p\in\D$, so we obtain true positive skewness.
	\end{proof}

	\begin{proof}[Proof of Theorem \ref{criteria}]
		Let $f$ denote the density function of $X$. For $p\in\D\setminus\{0\}$, define
		$$h_p(c):=f(\nu_p+c)-f(\nu_p-c),\ c\in[0,\nu_p]$$
		and
		$$\s_p:=\{c>0:h_p(c)>0\}.$$
		If $\s_p$ is non-empty, then its infimum 
		\begin{equation}\label{cp}
			c_p:=\inf\s_p
		\end{equation}
		exists and is non-negative. Note that if $f$ is continuous, then so is $h_p$. Then $\s_p$ 
		is the preimage of an open set under $h_p$, so $\s_p$ is also open and $c_p\notin\s_p$. 
		Since $h_p(0)=0$, then continuity implies
		\begin{equation}\label{hp(cp)=0}
			h_p(c_p)=0.
		\end{equation}
		Similarly, if $f$ is differentiable, then so is $h_p$. Because $h_p(c)\le0$ for all $c<c_p$, then
		\begin{equation}\label{hp'(cp)>=0}
			h_p'(c_p)\ge0.
		\end{equation}
	
		\noindent To prove the theorem, we will require several lemmas.
	
	\begin{lemma}\label{criteria l1}
		 The density $f$ is convex on $(0,\theta_1)\cup(\theta_2,\infty)$, concave on 
		 $(\theta_1,\theta_2)$, strictly increasing on $(0,\nu_0)$, and strictly decreasing 
		 on $(\nu_0,\infty)$.
	\end{lemma}
	\begin{proof}
		Since $f$ is positive on its support, then it must be increasing on $(0,\nu_0)$. If $f$ 
		is concave on $(0,\theta_1)$, then it is convex on $(\theta_1,\theta_2)$, contradicting 
		the fact that $\nu_0\in(\theta_1,\theta_2)$. The convexity part of the lemma follows, and 
		since $\theta_2>\nu_0$ and $f$ is unimodal, then $f$ is decreasing on $(\nu_0,\infty)$. 
		Moreover, since $f'>0$ near 0 and $f'<0$ near infinity, then $f'$ has an odd number of 
		zeroes. Integrability of $f$ implies $f'\to0$, hence $f'$ is non-zero everywhere on 
		$(0,\infty)$ except at $\nu_0$, otherwise $f$ would have at least three inflection points.
	\end{proof}
	
	\noindent Since $f$ is $C^2$ on $(0,\infty)$, then \eqref{hp(cp)=0} and \eqref{hp'(cp)>=0} hold. Moreover, 
	$h_p$ is $C^2$. Let 
	$$\D^*:=\{p\in\D\setminus\{0\}:\nu_p>(\nu_0+\theta_2)/2\}. \footnote{\ If, as in Remark 
		\ref{criteria relaxations}, we use 
		the condition $\nu_1+c_1>\theta_2$ rather than $\nu_1>(\nu_0+\theta_2)/2$, then instead let 
		$\D^*:=\{p\in\D:\nu_p\ge\nu_1+c_1\}$. The arguments that follow still apply with trivial 
		modifications.}
	$$
	By assumption $1\in\D^*$. We show that membership in $\D^*$ is sufficient to make \eqref{hp'(cp)>=0} 
	a strict inequality.
	
	\begin{lemma}\label{criteria l2}
		$h_p'(c_p)>0$ for all $p\in\D^*$.
	\end{lemma}
	\begin{proof}
		First we locate $\nu_p+c_p$ and $\nu_p-c_p$. Note that $f(\nu_p-c)$ is increasing and $f(\nu_p+c)$ 
		is decreasing for $c\in(0,\nu_p-\nu_0)$. The fact that $h_p$ is negative near 0 and $h_p(c_p)=0$ 
		implies $c_p>\nu_p-\nu_0$. By assumption $\nu_p>(\nu_0+\theta_2)/2$, so it follows that 
		\begin{equation}\label{nup+cp>theta}
			\nu_p+c_p>\theta_2.
		\end{equation}
		On the other hand, we easily have $\nu_p-c_p>0$, otherwise $h_p(c_p)=f(\nu_p+c_p)>0$. If 
		$\nu_p-c_p>\nu_0$, then $f(\nu_p-c_p)>f(\nu_p+c_p)$ by Lemma \ref{criteria l1}, again contradicting 
		$h_p(c_p)=0$. Thus
		\begin{equation}\label{nup-cp<nu0}
			0<\nu_p-c_p<\nu_0.
		\end{equation}
		Now suppose for the sake of contradiction that $h'_p(c_p)=0$.  Either $\nu_p-c_p\in(0,\theta_1)$ or 
		$\nu_p-c_p\in[\theta_1,\nu_0)$ by \eqref{nup-cp<nu0}. If the latter holds, then $f''(\nu_p-c_p)<0$, 
		and \eqref{nup+cp>theta} implies $f''(\nu_p+c_p)>0$. Note that 
		$h''_p(c_p)=f''(\nu_p+c_p)-f''(\nu_p-c_p)$, hence $h''_p(c_p)>0$. By continuity of $h''_p$ and 
		\eqref{hp(cp)=0}, $h_p$ is strictly convex and thus positive in a neighborhood of $c_p$ (excluding 
		$c_p$ itself), contradicting the minimality of $c_p$ in $\s_p$.
	
		Suppose instead that $\nu_p-c_p\in(0,\theta_1)$. Then from conditions (1) and (2) and equations 
		\eqref{nup+cp>theta} and \eqref{nup-cp<nu0}, we have the inequalities $\nu_0f'(\nu_p-c_p)>
		f(\nu_p-c_p)$ and $\nu_0f'(\nu_p+c_p)>-f(\nu_p+c_p)$. It follows that $\nu_0h'_p(c_p)>-h_p(c_p)=0$
		and we obtain a contradiction.
	\end{proof}
	
	\noindent Next, we prove some properties of $c_p$.
	
	\begin{lemma}\label{criteria l3}
		The map $p\mapsto c_p$ is continuously differentiable in $\D^*$.
	\end{lemma}
	\begin{proof}
		Define $\psi:\D^*\times\R_+$ by $\psi(p,c):=h_p(c)$. By Proposition \ref{nup C1}, $\psi$ is 
		differentiable and has partial derivatives
		$$\frac{\partial}{\partial c}\psi(p,c)=f'(\nu_p+c)+f'(\nu_p-c)=h'_p(c),$$
		$$\frac{\partial}{\partial p}\psi(p,c)=(f'(\nu_p+c)-f'(\nu_p-c))\frac{d\nu_p}{dp},$$
		both of which are jointly continuous in $p$ and $c$. Thus $\psi$ is continuously differentiable.
		By Lemma \ref{criteria l2}, $\frac\partial{\partial c}\psi(p,c_p)>0$ for all $p\in\D^*$ and so the 
		continuous differentiability of $p\mapsto c_p$ follows from \eqref{hp(cp)=0} and the implicit 
		function theorem.
	\end{proof}
		
	\begin{lemma}\label{criteria l4}
		For any $p\in\D^*$, $\frac{dc_p}{dp}$ and $\frac{d\nu_p}{dp}$ have the same sign, and 
		$\big|\frac{dc_p}{dp}\big|>\big|\frac{d\nu_p}{dp}\big|$ if they are non-zero.
	\end{lemma}
	\begin{proof}
		By \eqref{hp(cp)=0}, we have $f(\nu_p+c_p)-f(\nu_p-c_p)=0$. The left side is differentiable in $p$ 
		by Lemma \ref{criteria l3}, so taking derivatives, we obtain
		\begin{equation}\label{dp l4}
			f'(\nu_p+c_p)\left(\frac{d\nu_p}{dp}+\frac{dc_p}{dp}\right)
				- f'(\nu_p-c_p)\left(\frac{d\nu_p}{dp}-\frac{dc_p}{dp}\right) = 0.
		\end{equation}
		Rearranging yields
		$$\frac{dc_p}{dp} = \frac{d\nu_p}{dp}
			\left(\frac{f'(\nu_p-c_p)-f'(\nu_p+c_p)}{f'(\nu_p-c_p)+f'(\nu_p+c_p)}\right),$$
		where the fraction is well-defined with positive denominator by Lemma \ref{criteria l2}. The 
		numerator is positive by Lemma \ref{criteria l1}, \eqref{nup+cp>theta}, and \eqref{nup-cp<nu0}, so 
		$\frac{dc_p}{dp}$ has the same sign as $\frac{d\nu_p}{dp}$. If $\frac{d\nu_p}{dp},\frac{dc_p}{dp}
		>0$, then one can see immediately from \eqref{dp l4} that $\frac{dc_p}{dp}>\frac{d\nu_p}{dp}$. 
		The reverse also follows.
	\end{proof}
	
	\noindent Our final lemma concerns a criterion for pointwise increasingness of $\nu_p$.
	
	\begin{lemma}\label{criteria l5}
		Fix $p\in\D\setminus\{0\}$. If
		\begin{equation}\label{slope condition}
			(\nu_p-c_p)\frac{f'(\nu_p+c_p)}{f(\nu_p+c_p)} > -1,
		\end{equation}
		then $\nu$ is increasing at $p$.
	\end{lemma}
	\begin{proof}
		Since $f(\nu_p+c_p)=f(\nu_p-c_p)$, we may rearrange to obtain
		$$-\frac{f(\nu_p-c_p)}{\nu_p-c_p}<f'(\nu_p+c_p).$$
		Define the line
		$$\ell(x):=-\frac{xf(\nu_p-c_p)}{\nu_p-c_p}+\frac{2\nu_pf(\nu_p-c_p)}{\nu_p-c_p},\ x\ge\nu_p+c_p,$$
		and note that
		\begin{equation}\label{line bds}
			\begin{split}
				\ell(\nu_p+c_p) &= f(\nu_p-c_p)=f(\nu_p+c_p), \\
				\ell(2\nu_p) &= 0=f(0).
			\end{split}
		\end{equation}
		Note that \eqref{slope condition} implies $f'(\nu_p+c_p)=\ell'(\nu_p+c_p)$, so by convexity of 
		$f$ on $(\theta_2,\infty)$, we have $f'>\ell'$ on $(\nu_p+c_p,\infty)$. Via integration we obtain 
		\begin{equation}\label{line below}
			f(\nu_p+c)>\ell(\nu_p+c)
		\end{equation}
		for $c>c_p$. We now split into two cases to show $f(\nu_p+c)>f(\nu_p-c)$ for $c>c_p$.
		
		\textbf{Case 1.} Suppose $\nu_p-c_p\in(0,\theta_1]$. By \eqref{line bds} and the convexity of 
		$f$ on this interval, we have $\ell(x)>f(2\nu_p-x)$ for $x\in(\nu_p+c_p,2\nu_p)$, making use of 
		the fact that convexity is preserved under reflection and translation. A change of variables 
		gives the inequality $\ell(\nu_p+c)>f(\nu_p-c)$ for $c\in(c_p,\nu_p)$. Combining with 
		\eqref{line below} yields $f(\nu_p+c)>f(\nu_p-c)$ for $c\in(c_p,\nu_p)$. If $c\ge\nu_p$, 
		then $f(\nu_p+c)>0=f(\nu_p-c)$ and we are done.
		
		\textbf{Case 2}. Suppose instead $\nu_p-c_p\in(\theta_1,\nu_0)$. If it happens that 
		$\ell(\nu_p+c)>f(\nu_p-c)$ for all $c\in(c_p,\nu_p)$, then the argument in Case 1 applies and we 
		are done. Otherwise, let
		$$\tilde\ell(x):=-(x-\nu_p-c_p)f'(\nu_p-c_p)+f(\nu_p+c_p),\ x\ge \nu_p+c_p,$$
		be the line such that
		$$\tilde\ell(\nu_p+c_p)=f(\nu_p+c_p)$$
		and
		$$\tilde\ell(\nu_p+c_p+f(\nu_p+c_p)/f'(\nu_p-c_p))=0.$$
		Lemma \ref{criteria l2} and convexity imply $f'(\nu_p+c)>\tilde\ell'(\nu_p+c)$ and thus 
		$f(\nu_p+c)>\tilde\ell(\nu_p+c)$ for all $c>c_p$. By concavity of $f$ near $\nu_p-c_p$, one 
		can easily see that $\tilde\ell(\nu_p+c)>f(\nu_p-c)$ for $c\in(c_p,\nu_p-\theta_1]$. It 
		remains to show $\tilde\ell(\nu_p+c)>f(\nu_p-c)$ for $c\in(\nu_p-\theta_1,\nu_p)$.
		
		We have by assumption that $f(\nu_p-c)>\ell(\nu_p+c)$ for $c$ in a right neighborhood of $c_p$.
		We also have by \eqref{line bds} that $f(\nu_p-c_p)=\ell(\nu_p+c_p)$. It follows that 
		$$-f'(\nu_p-c_p)>\ell'(\nu_p+c_p)=-\frac{f(\nu_p-c_p)}{\nu_p-c_p}.$$
		Substituting with $f(\nu_p-c_p)$ with $f(\nu_p+c_p)$ and rearranging yields
		$$\nu_p+c_p+\frac{f(\nu_p+c_p)}{f'(\nu_p-c_p)}>2\nu_p.$$
		The left side is precisely the root of $\tilde\ell$ whereas the right side is a root of $f$. 
		We showed previously that $\tilde\ell(2\nu_p-\theta_1)>f(\theta_1)$. Convexity of $f$ on 
		$(0,\theta_1)$ implies $\tilde\ell(\nu_p+c)>f(\nu_p-c)$ for $c\in(\nu_p-\theta_1,\nu_p)$, 
		and we are done.
		
		We have now shown that, in general, $f(\nu_p+c)>f(\nu_p-c)$ for $c>c_p$. We also know by 
		definition of $c_p$ that $f(\nu_p+c)\le f(\nu_p-c)$ for $c<c_p$. Thus the conditions of Lemma 
		\ref{Kovchegov L2} are satisfied and so $\nu$ is increasing at the point $p$.
	\end{proof}
	
		\noindent To prove the main theorem, it suffices to show \eqref{slope condition} for all $p\in\D^*$ 
		and that $\D^*=\D\setminus\{0\}$. If we have conditions (1) and (2) as they are written, then \eqref{slope condition} 
		holds for all $p\in\D^*$ immediately by \eqref{nup+cp>theta} and \eqref{nup-cp<nu0}. Suppose instead we only 
		have the weaker condition as stated in Remark \ref{criteria relaxations}(b):
		\begin{equation*}
			\begin{split}
				f'/f > 1/(\nu_1-c_1) &\text{ on } (0,\theta_1), \\
				f'/f > -1/(\nu_1-c_1) &\text{ on } (\nu_1+c_1,\infty).
			\end{split}
		\end{equation*}
		Let $u(p):=\nu_p-c_p$ on $\D^*$, so $u$ is differentiable on its domain. Also note that $u>0$ by 
		\eqref{nup-cp<nu0}. Suppose for some $p'\in\D^*$ that $u(p')\le\nu_1-c_1$. Then
		$$(\nu_{p'}-c_{p'})\frac{f'(\nu_{p'}+c_{p'})}{f(\nu_{p'}+c_{p'})}>-1,$$
		i.e., \eqref{slope condition} holds for $p'$ and so $\nu_p$ is increasing at $p'$. Lemma 
		\ref{criteria l4} implies $u$ is decreasing at $p'$. It follows that for all $p\ge p'$, $u$ is 
		decreasing and so $u(p)\le\nu_1-c_1$. In particular, $\nu_p$ is increasing for all $p\ge p'$.
		
		By assumption, $1\in\D^*$ and trivially $u(1)\le\nu_1-c_1$. Since $\D^*\in[1,\infty)$, then 
		$\nu_p$ is increasing for all $p\in\D^*$. Recall that $p\in\D^*$ if $\nu_p>(\nu_0+\theta_2)/2$. 
		Since $\nu_p$ is increasing at $p=1$ and for all $p\in\D^*$, then $\nu_p>(\nu_0+\theta_2)/2$ 
		for all $p\in\D\setminus\{0\}$, hence $\D\setminus\{0\}=\D^*$.\footnote{\ Clearly the argument 
		for $\D^*=\D\setminus\{0\}$ still applies if we use the alternative definition for $\D^*$.}
	\end{proof}
	
	\begin{proof}[Proof of Corollary \ref{criteria 2}]
		The proof directly extends to the case where $f''$ only has a single root $\theta>\nu_0$ by 
		setting $\theta_1=0$. In fact, conditions (1) and (2) are not necessary. Indeed, we use them 
		once in the proof of Lemma \ref{criteria l2} in the case $\nu_p-c_p\in(0,\theta_1)$, but this 
		is no longer relevant if $f''$ has only one root. The only other time we use the conditions 
		is to prove that \eqref{slope condition} holds for all $p\in\D^*$. However, if $f''$ has only
		one root then, \eqref{slope condition} holds automatically. Indeed, note that $f$ is concave 
		and increasing near $\nu_p-c_p$ for all $p\in\D^*$, hence
		$$\frac{f(\nu_p-c_p)}{\nu_p-c_p}>f'(\nu_p-c_p).$$
		Since $f'(\nu_p+c_p)<0$, then by substituting $f(\nu_p+c_p)=f(\nu_p-c_p)$, we have
		$$(\nu_p-c_p)\frac{f'(\nu_p+c_p)}{f(\nu_p+c_p)}>\frac{f'(\nu_p+c_p)}{f'(\nu_p-c_p)}.$$
		The right side dominates $-1$ as a consequence of Lemma \ref{criteria l2}, so we arrive at 
		\eqref{slope condition}. This proves the corollary.
	\end{proof}
		

\section{Discussion}
This work attempts to broaden our understanding of true skewness by demonstrating the true skewness of several additional distributions and by establishing simpler criteria for which one can conclude a distribution is truly skewed. Theorems \ref{Levy TPS}--\ref{Skew-normal TPS}, which establish the parameter regions for which the L\'evy, chi-squared, Weibull, and skew-normal distributions are truly positively skewed, all rely crucially on showing that the conditions of Lemma \ref{Kovchegov L2} hold for every relevant value of $p$. This is done by analyzing what we have called the ``log density ratio of the left and right parts,'' which for given $p$ is
$$R_p(x) = \log\left(\frac{f(\nu_p-x)}{f(\nu_p+x)}\right), \qquad x\in[0,\nu_p),$$
where $f$ is the density function of the distribution in question. The conditions of Lemma \ref{Kovchegov L2} hold if one can show that $R_p(\cdot)$ has exactly one positive root $x_0$, satisfying $R_p(x)>0$ for $0<x<x_0$ and $R_p(x)<0$ for $x>x_0$. In Theorems \ref{Levy TPS}, \ref{chi-squared TPS}, and \ref{Weibull TPS}, this can be shown in quite a straightforward manner simply by finding the critical points of $R_p(\cdot)$. Such an approach may likely be applied to other visibly skewed distributions for which a closed-form expression for its density function exists.

However, for the skew-normal distribution, Theorem \ref{Skew-normal TPS}, the absence of a closed-form expression for its density function makes the computation of the critical points of $R_p(\cdot)$ intractible. Instead, we exploit the fact that the log derivative of a distribution function $\Psi$ is simply $\psi/\Psi$, where $\psi$ is the corresponding density function. The ratio $\psi/\Psi$ is the reciprocal of what is commonly referred to as the {\em hazard rate} of the distribution $\Psi$, which has been studied with some detail for most well-known distributions. In the proof of Theorem \ref{Skew-normal TPS}, we encounter the reciprocal of the hazard rate of the standard Gaussian, which goes by the special name of the {\em Mills' ratio}. The Mills' ratio, as one may expect, has been very well-studied, allowing us to employ a closed-form expression, which very tightly bounds the quantities we care about, to prove the required properties of the skew-normal distribution's log density ratio of the left and right parts.

Indeed, the proof of Theorem \ref{Skew-normal TPS} identifies a key yet unsurprising connection between the notion of true skewness, which is fundamentally a comparison of the rate of decay of the two sides of a distribution, and the hazard rate. It suggests that the notion of true skewness is an accurate reflection, and in a much more rigorous fashion than Pearson's coefficients of skewness, of the essence of what we imagine {\em skewness} to mean. More practically, the method by which we prove Theorem \ref{Skew-normal TPS} should work for a larger variety of distributions; a reasonable first direction would be skewed versions of other symmetric distributions, as introduced in, for example, Chapter 1 of \cite{azzalini2013skew}.

The remainder of this section discusses sums and products of truly skewed random variables; examines whether true skewness extends to the discrete case; and provides an interpretation of true skewness for multivariate distributions. 

\subsection{Sums and products of truly skewed random variables}

The goal of this paper and this discussion is to present tools for establishing a larger class of distributions which are truly skewed. A natural question is whether or not true skewness is preserved under certain ``transformations'' of truly skewed random variables. 

Theorem \ref{u(X)} gives an affirmative answer to this question, but with limitations: transforming a random variable with a strictly decaying distribution by an increasing convex function preserves its true positive skewness. One might also expect that sums of truly positively skewed random variables preserve true positive skewness, but the rather strong requirement that the $p$-means of a truly positively skewed random variable be {\em strictly} increasing allows one to construct rather straightforward counterexamples. Indeed, we may even take the two summands to be identically distributed. For $1/2<\lambda<1$, let $X$ and $Y$ be i.i.d. with density
\begin{equation}\label{counterexample density}
	f(x):=\begin{cases}
		\lambda & 0 < x \le 1 \\
		1-\lambda & 1 \le x < 2 \\
		0 & \text{otherwise}.
	\end{cases}
\end{equation}
Since $f$ is decreasing on its support, $X$ and $Y$ are truly positively skewed by Proposition \ref{relaxed Kovchegov P2}. One can easily compute the density function $f_Z$ of $Z=X+Y$, the convolution of $f$ with itself and see that, depending on the value of $\lambda$, the $p$-means of $Z$ are not always increasing. For instance, for $\lambda=3/5$, computing numerically the median $\nu_1^Z$ and the integrals in \eqref{dnup numerator} yields $\nu_1^Z\approx 1.786$ and
$$\int_0^{4-\nu_1}\log y\ f_Z(\nu_1+y)dy - \int_0^{\nu_1} \log y\ f_Z(\nu_1-y)dy\approx-0.000699 < 0,$$
which implies that $p\mapsto\nu_p^Z$ is decreasing at $p=1$, and so $Z$ is not truly positively skewed.\footnote{\ Numerical computations were carried out using Wolfram Mathematica, version 12.3.0.0.}

On the other hand, it may be more fruitful to start with random variables with monotone density functions. One class of truly positively skewed densities that is closed under summation is the class of ``decreasing linear densities,'' of the form
$$f(x) = h - \frac{h^2x}{2}, \qquad x\in[0,2/h],$$
for arbitrary $h>0$.

\subsection{True skewness of discrete distributions}

In the discrete case, even the simplest sums of truly skewed random variables fail to remain truly skewed. Take $X$ and $Y$ to be independent $\Bernoulli(1/3)$ random variables; one can show that these are truly positively skewed quite easily. Their sum is $\Binom(2,1/3)$, which has median 1 and mean 2/3, hence $\nu_2<\nu_1$, and true positive skewness fails to hold. Indeed, the notion of true skewness is finnicky for discrete distributions, at least in the way it has been formulated here and in \cite{Kovchegov21}: this could be due to the interpretation of discrete distributions as limits of successfully sharper (continuous) bump functions, in which case distributions like the $\Binom(n,1/3)$ are no longer unimodal but in fact have $n$ modes.
	
	The proof techniques used in this paper often involve taking the log of a ratio of the density function 
	of a distribution. In particular, the logarithm of a product of two entities is equal to the sum of the individual 
	logarithms of those entities. For this reason,  it seems natural to examine how true skewness behaves under 
	the products of random variables and whether it is preserved. In particular, it seems that taking the logarithm 
	of the density of product of two random variables would yield a sum which could then imply true skewness 
	depending on how one defines the two random variables. \\
	
%
%


\subsection{True skewness of continuous multivariate distributions}\label{extended settings section}
The definition of \textbf{Fr\'echet $p$-mean} (Def. \ref{nup def 1}) extends naturally to multivariate distributions as follows.

\noindent Let $\mathbf{X} \in \mathbb{R}^k$ be a random vector. The $p$-mean $\boldsymbol{\nu}_p = \left(\nu_p^{(1)}, \nu_p^{(2)}, \ldots, \nu_p^{(k)}\right)$ of $\mathbf{X}$ is defined 
\begin{equation}\label{multivariate nu_p}
	\boldsymbol{\nu}_p = \argmin_{\mathbf{a} \in \mathbb{R}^k} E \big[\|\mathbf{X} - \mathbf{a}\|^{p}\big],
\end{equation}
where $\|\cdot\|$ is the usual Euclidean norm. 

In the univariate setting, true skewness corresponds to the sign of $d\nu_p/dp$ representing the direction of trajectory $\boldsymbol{\nu}_p$. 
Following \cite{Kovchegov21}, we adjust true skewness accordingly. We let 
$$\boldsymbol{\tau}_p := 
\frac{d\boldsymbol{\nu}_p}{dp} \Big/ \left\| \frac{d\boldsymbol{\nu}_p}{dp} \right\|$$
denote the unit tangent vector for trajectory of $\boldsymbol{\nu}_p$ in $\mathbb{R}^k$. We will illustrate true skewness in a multivariate setting by means of an example.
It was conjectured in \cite{Kovchegov21} that the limiting direction vector $\lim\limits_{p \to \sup \mathcal{D}} \boldsymbol{\tau}_p$ with $p$ increasing to the rightmost bound in its domain $\mathcal{D}$ may be interpreted as the ``direction'' of true skewness.

As an example, considered a {\it multivariate skew normal distribution} defined in \cite{lachos2014multivariate}, the  probability density function of a multivariate skew-normal random vector $\mathbf{Y} \in \mathbb{R}^k$ is 
\begin{equation}\label{MVSN density}
	f(\mathbf{y}) = 2 \phi_k(\mathbf{y}; \boldsymbol{\mu}, \boldsymbol{\Sigma}) \Phi_1(\boldsymbol{\lambda}^\top \boldsymbol{\Sigma}^{-1/2}(\mathbf{y} - \boldsymbol{\mu})) \qquad \mathbf{y} \in \mathbb{R}^k
\end{equation}
where $\phi_k(\cdot\,;\boldsymbol{\mu}, \boldsymbol{\Sigma})$ is the density function of the $k$-variate normal distribution with mean  $\boldsymbol{\mu}$ and covariance matrix $\boldsymbol{\Sigma}$, and $\Phi_1(\cdot)$ is the cumulative distribution function of the univariate standard normal distribution. Taking $\boldsymbol{\lambda} = 0$ recovers the standard multivariate normal distribution with density $\phi_k(\cdot; \boldsymbol{\mu}, \boldsymbol{\Sigma})$ . We refer to $\boldsymbol{\lambda}$ as the \textit{skewness parameter} vector.
\begin{figure}[h]
	\centering
	\begin{subfigure}[b]{0.30\textwidth}
		\centering
		\includegraphics[width=\textwidth]{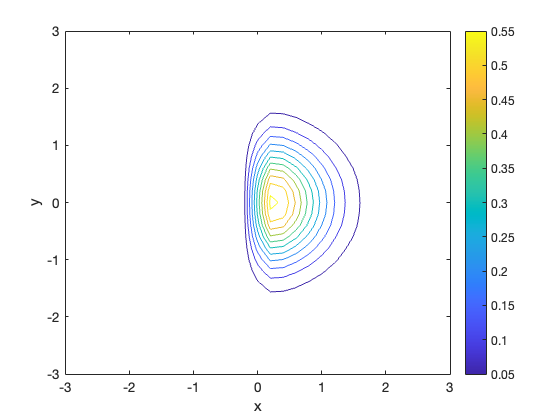}
		\caption{$\boldsymbol{\lambda} = (5,0)^\top$}
	\end{subfigure}
	\begin{subfigure}[b]{0.30\textwidth}
		\centering
		\includegraphics[width=\textwidth]{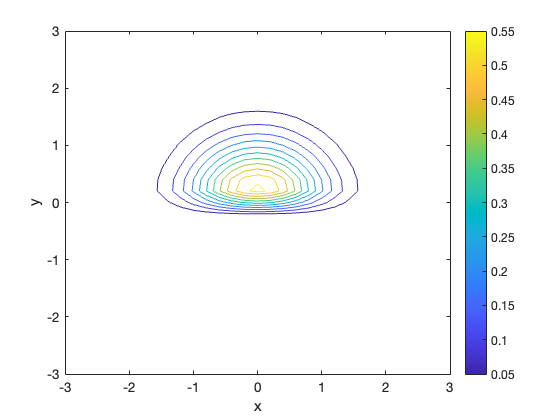}
		\caption{$\boldsymbol{\lambda} = (0,5)^\top$}
	\end{subfigure}
	\begin{subfigure}[b]{0.30\textwidth}
		\centering
		\includegraphics[width=\textwidth]{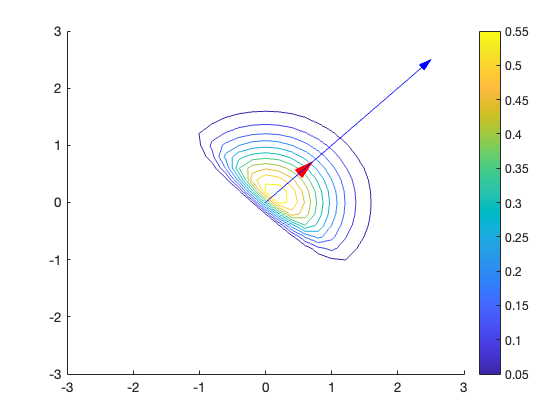}
		\caption{$\boldsymbol{\lambda} = (5,5)^\top$}
	\end{subfigure}
	\caption{Contour plots of 2-variate  skew-normal distributions. Red arrow is $\boldsymbol{\tau}_p$, which is the same for all $p$, and blue arrow is $\boldsymbol{\lambda}$.}
	\label{fig: mvsn distribution plots}
\end{figure}

The arrows plotted in Figure \ref{fig: mvsn distribution plots} motivate the interpretation of multivariate true skewness via direction vectors $\boldsymbol{\tau}_p$, which are naturally co-linear with the skewness parameter vector $\boldsymbol{\lambda}$. 

%

\section*{Acknowledgments}
The majority of this research was done as a part of a Research Experience for Undergraduates (REU) program at the Oregon State University Department of Mathematics funded by NSF grant DMS-1757995. We would like to thank Javier Rojo of Indiana University for his helpful feedback on the concepts considered in this paper. We would also like to acknowledge professor Holly Swisher of Oregon State University who as the director of the REU program provided invaluable perspective throughout the summer.


\end{document}